\documentclass[11pt,a4paper]{article} 
\usepackage[T1]{fontenc}
\usepackage{lmodern}
\usepackage{varioref}
\usepackage{indentfirst}
\usepackage[english]{babel}
\usepackage{amsmath}
\usepackage{amsthm}
\usepackage{amsfonts}
\usepackage[vmargin=2.5cm,hmargin=3.3cm]{geometry}
\usepackage{amssymb}
\usepackage{pstricks-add}
\usepackage{graphicx} 
\usepackage{color}
\usepackage{wallpaper}
\usepackage[english]{minitoc}
\usepackage{fancyhdr} 
\color{black}
\usepackage{fancyhdr}
\usepackage{url}
\usepackage{hyperref} 
\hypersetup{
colorlinks=true, 
breaklinks=true, 
urlcolor= blue, 
linkcolor= red, 
citecolor=green, 
pdftitle={}, 
pdfauthor={Lahoucine Elaissaoui}, 
pdfsubject={Simulation} 
}

\theoremstyle{plain}
\newtheorem{theorem}{Theorem}[section]

\newtheorem{corollary}[theorem]{Corollary}
\newtheorem{remark}[theorem]{Remark}
\theoremstyle{definition}

\title{Fourier Expansion of the Riemann zeta function and applications }
\author{\textsc{BY} \\ \textsc{Lahoucine Elaissaoui}\footnote{\texttt{E-mail:lahoumaths@gmail.com}} \qquad \ \quad \textsc{And} \quad \textsc{Zine El Abidine Guennoun}\footnote{E-mail: guennoun@fsr.ac.ma} \\ \\ \\ \sc{Mohammed V University in Rabat} \\ \sc{Faculty of Sciences} \\ \sc{Department of Mathematics}  \\  \\ \sc{Morocco}  }

\date{August 24, 2017 revised 2019}

\begin{document}

\maketitle

\begin{abstract}
We study the distribution of values of the Riemann zeta function $\zeta(s)$ on vertical lines $\Re s + i \mathbb{R}$, by using the theory of Hilbert space. We show among other things, that, $\zeta(s)$ has a Fourier expansion in the half-plane $\Re s \geq 1/2$ and its Fourier coefficients are the binomial transform involving the Stieltjes constants. As an application, we show explicit computation of the Poisson integral associated with the logarithm of $\zeta(s) - s/(s-1)$. Moreover, we discuss our results with respect to the Riemann and Lindel\"{o}f hypotheses on the growth of the Fourier coefficients.

\end{abstract}





\section{Introduction and statement of results}
\label{}
The Riemann zeta function is of great interest in number theory since its value-distribution as a complex function may decode, among others, relevant information on prime numbers. The Riemann zeta function is defined by $\zeta(s) := \sum_{n \geq 1}1/n^s$ for $\Re s > 1$ and has an analytic continuation to the whole complex plane except for a simple pole at $s=1$. The Laurent expansion of $\zeta(s)$ near its pole is given by 
\begin{equation}
\zeta(s) = \frac{1}{s-1} + \sum_{k = 0}^{\infty} \frac{\gamma_k}{k!} (1-s)^k; 
\label{LauExp}\end{equation}
where 
\begin{equation}
\gamma_k := \lim_{N \to \infty}\left( \sum_{m=1}^N \frac{\log^km}{m} - \frac{\log^{k+1}N}{k+1} \right)
\label{Stlcte}\end{equation}  
 are the so-called Stieltjes constants and $\gamma_0$ is the Euler-Mascheroni constant; see for example \cite{Ber}. It should be noted that, neither the explicit form of $\gamma_k$ nor its algebraic properties are known. However, Briggs \cite{Brig} proved that infinitely many $\gamma_n$ are positive, and infinitely many are negative. At this point, it is worth highlighting an important
and far reaching estimate of Berndt \cite{Ber} in the form
\begin{equation}
\frac{|\gamma_k|}{k!} \leq\frac{3+(-1)^k}{k \pi^k} \leq \frac{4}{k\pi^k} 
\label{UpStiel}
\end{equation}
which lead to some pleasant recent developments. Concerning these developments, Alkan \cite{Alk} settles, by using \eqref{LauExp} and \eqref{UpStiel}, a conjecture of Cerone and Dragomir on the concavity of the reciprocal of the Riemann zeta function on $(1,\infty).$

The value-distribution of the Riemann zeta function inside and on the boundary of the so-called critical strip $0<\Re s < 1$ is not yet understood completely, such as the distribution of zeros and the order of magnitude: The famous open Riemann hypothesis claims that all the nontrivial zeros of $\zeta(s)$, are denoted by $\varrho$, lie on the critical line, $\Re s = 1/2$; however, it is known that positive proportion, $\kappa$, of these zeros is on the critical line, we briefly mention the work of Levinson \cite{Lev} ($\kappa \geq 34.74\%$), Conrey \cite{Con} ($\kappa \geq 40.88\%$), Bui et al. \cite{Bui} ($\kappa \geq 41.05\%$) and Feng \cite{Fen} ($\kappa \geq 41.28\%$). Concerning the order of the Riemann zeta function, it is well-known that  $\zeta(s)$ is bounded in any half-plane $\Re s \geq \delta > 1$. Therefore, the order of $\zeta(s)$ for $\Re s \leq - \delta < 0$ follows from the following functional equation \cite[p. 16]{Titch}
\begin{equation}
\zeta(s) = \chi(s) \zeta(1-s),\label{FE}
\end{equation} 
where
$$\chi(s) = \pi^{s-\frac{1}{2}}\frac{\Gamma\left( \frac{1-s}{2} \right)}{\Gamma\left( \frac{s}{2}\right)} $$
and $\Gamma(s)$ is the well-knwon Euler gamma function. Please notice that Stirling's formula yields, for any bounded $\sigma$ and $|t|\geq 2$ ($s=\sigma+it$), 
$$\chi(s) = \left|\frac{t}{2\pi}\right|^{\frac{1}{2}-\sigma}\exp\left(i\left(\frac{\pi}{4}- t\log \frac{t}{2\pi e}\right)\right)\left\{1+ O\left( \frac{1}{|t|}\right)\right\}.$$
Moreover, by the Phragm\'{e}n-Lindel\"{o}f principle \cite[\S 9.41]{Titch2} one can deduce that if $\zeta\left( \frac{1}{2} + it \right) = O\left( t^{\lambda+ \varepsilon}\right)$, for any $\varepsilon > 0$, then we have uniformly in the strip $1/2 \leq \sigma < 1$, 
$$ \zeta(s) = O\left( t^{2\lambda(1 - \sigma)+ \varepsilon}\right), \qquad \forall \varepsilon > 0 .$$
Hence, the growth rate of $\zeta(s)$ in the rest of the critical strip follows from the functional equation \eqref{FE}. It should be noted that the present best optimal value of $\lambda$ is due to Bourgain \cite{Bour} who obtained the exponent $\lambda =13/84$. However, the yet unproved Lindel\"{o}f hypothesis states that the best estimate is for $\lambda = 0$. According to our study; we would like to point out that the Lindel\"of hypothesis is related to the probability measure $\mu$ defined on $\mathcal{B}$, the Borel sigma-algebra associated with the real numbers set $\mathbb{R}$, by

\begin{equation}
\mu(A) := \frac{1}{2\pi} \int_{A} \frac{\mathrm{d} t}{\frac{1}{4} + t^2} \ , \qquad \forall A \in \mathcal{B}.
\label{cauchy}
\end{equation}
Actually, the measure $\mu$ appears in various number-theoretical interpretations, especially, in term of ergodic and probability theories as we show in the last section. 

The main purpose of this paper is to show that the distribution of values of $\zeta(s)$ ($s=\sigma + it$) in the half-plane $\sigma > 1/2$ and on its boundary depends on the Stieltjes constants \eqref{Stlcte}. Namely, on the critical line we have the following Fourier expansion of the Riemann zeta function.

\begin{theorem}
For any $t \in \mathbb{R}$, we have
\begin{equation}
\zeta\left( \frac{1}{2} + it\right) =  \gamma_0 - \frac{1}{\frac{1}{2}- it} + \sum_{n \geq 1}  \ell_n \left( \frac{\frac{1}{2} - it}{\frac{1}{2}+it} \right)^n;
\label{Zeta12}\end{equation}
where 
$$ \ell_n := (-1)^n \sum_{k=1}^n \binom{n-1}{k-1} \frac{(-1)^k \gamma_k}{k!}$$
for any integer $n \geq 1$.
\label{Th1}\end{theorem}

The proof of Theorem \ref{Th1}, as we show in the next section, is based on technics of the functional analysis. For this reason, we recall that the class, which is usually denoted by $\mathrm{L}^2(\mu),$ of all complex-valued functions $f$ defined on the measure space $(\mathbb{R}, \mathcal{B}, \mu)$ such that
$$\left\| f \right\|_2^2 := \int_{\mathbb{R}}|f(t)|^2 \mathrm{d}\mu(t)< + \infty ,$$
forms a Hilbert space. We would like to point out that $\mathrm{L}^2(\mu)$ really consists of equivalence classes of functions; because, this makes $\|\cdot\|_2$ a norm. Indeed, for $f,g \in \mathrm{L}^2(\mu)$ we have, $\|f-g \|_2 = 0$ if and only if $f=g$ almost everywhere with respect to the measure $\mu .$ 
Moreover, the sequences of functions defined for any integer $n$ by 
$$ e_n(t) := \left( \frac{\frac{1}{2}-it}{\frac{1}{2}+it}\right)^n , \qquad (t \in \mathbb{R})  $$
form an orthonormal basis for the Hilbert space $\mathrm{L}^2(\mu)$ associated with the following inner product
$$ \langle f , g\rangle := \int_{\mathbb{R}} f(t) \overline{g(t)} \mathrm{d}\mu(t)$$
for any complex-valued functions $f$ and $g$ in $\mathrm{L}^2(\mu)$; where $\overline{g}$ denotes the complex conjugate of $g$. In fact, the set $\{e_n\}$  is a Fourier basis, since $e_n(t) = \exp(-2in \arctan(2t))$ for any real $t$ and every integer $n$; here $\arctan$ denotes the arctangent function (inverse tangent function) which is bijective from $\mathbb{R}$ to $(-\pi/2 , \pi/2)$. Thus, for any integers $n$ and $m$ we have

\begin{align*}
\langle e_n, e_m\rangle &= \frac{1}{2\pi} \int_{\mathbb{R}} \left( \frac{\frac{1}{2}-it}{\frac{1}{2}+it} \right)^{n-m}\frac{\mathrm{d}t}{\frac{1}{4}+ t^2} \\ &= \frac{1}{2\pi } \int_{-\pi}^{\pi} e^{i(n-m)\theta}\mathrm{d}\theta \\ &= \begin{cases}\begin{array}{lc} 1 & \text{if} \ n=m \\ 0 & \text{otherwise} .\end{array}\end{cases}
\end{align*}
The second equality follows by the substitution $t = \tan(\theta/2)/2$. Therefore, any complex-valued function $f \in \mathrm{L}^2(\mu)$ can be expanded as 
$$f = \sum_{n \in \mathbb{Z}} c_n e_n, $$
where $\mathbb{Z}$ denotes the set of integers; $\{c_n\}$ are the so-called Fourier coefficients of $f$ and are given by
$$c_n = \langle f,e_n \rangle = \frac{1}{2\pi } \int_{\mathbb{R}} f(t) \overline{e_n(t)} \mathrm{d}\mu(t).$$

Hence, by exploiting this mathematical background and some technics of analysis, we provide in Section \ref{Section2} the proof of Theorem \ref{Th1}. By a similar reasoning, we obtain the expansion of $\zeta(s)$ in the Hilbert space $\mathrm{L}^2(\mu)$ on every vertical line $\Re s = \sigma_0 > 1/2$, except on the vertical line $\Re s = 1$. Namely,

\begin{theorem}
Let $\sigma_0 > 1/2$. Then for any $t \in \mathbb{R},$ we have
\begin{itemize}
\item If $\sigma_0 < 1$,
$$\zeta(\sigma_0 + it) = \zeta\left(\sigma_0+\frac{1}{2}\right)- \frac{1}{\sigma_0 - \frac{1}{2}} + \frac{1}{\sigma_0 - 1 + it} + \sum_{n \geq 1} \ell_n(\sigma_0) e_n(t),$$ 
where
$$\ell_n(\sigma_0) := (-1)^n \sum_{k=1}^n\binom{n-1}{k-1}\left(\frac{1}{k!}\zeta^{(k)}\left( \sigma_0 + \frac{1}{2}\right) - \frac{(-1)^k}{\left(\sigma_0 - \frac{1}{2} \right)^{k+1}}  \right),  $$
for any integer $n \geq 1$.
\item If $\sigma_0 > 1$,
$$\zeta(\sigma_0 + it) = \zeta\left(\sigma_0 + \frac{1}{2} \right) + \sum_{n \geq 1} \ell_{n}(\sigma_0) e_n(t); $$
where in this region,
$$\ell_n(\sigma_0) := (-1)^n \sum_{k=1}^n\binom{n-1}{k-1}\frac{1}{k!}\zeta^{(k)}\left( \sigma_0 + \frac{1}{2} \right). $$
\end{itemize}
$\zeta^{(k)}$ denotes the $k$th derivative of $\zeta$.
\label{Th2}\end{theorem}

 In Section \ref{Section3}, we show that the power series 
 \begin{equation} 
 \sum_{n \geq 0} \ell_n z^n;
 \label{Pseries} 
 \end{equation} 
 where $\ell_0=\gamma_0 - 1$ and $(\ell_n)_{n\geq 1}$ are the Fourier coefficients of the Riemann zeta function given in Theorem \ref{Th1}, is absolutely convergent, for any complex number $z$ in the open unit disk $\mathbb{D},$ to the function
\begin{equation}
h(z) = \frac{1}{z} + \zeta\left( \frac{1}{1+z} \right).
\label{UAC}
\end{equation}
 Therefore, it follows by Theorem \ref{Th1} that the power series \eqref{Pseries} converges for all points on the boundary of $\mathbb{D}$, except for $z=-1.$ Please notice that the series \eqref{Pseries} does not converge absolutely on the boundary of $\mathbb{D}$. However, as we remark in Section \ref{Section2}, the related series \eqref{Zeta12} converges compactly on $\mathbb{R}.$
 
 In the same section, we show that $h$ belongs to the classical Hardy space $\mathrm{H}^2(\mathbb{D})$; see for example \cite[Ch. 17]{Rud} ($p=2$). Thus, by \cite[Th.17.11]{Rud} the nontangential limits of $h$ exists almost everywhere on the unit circle $\mathbb{T} = \partial \mathbb{D}$ (boundary of $\mathbb{D}$); namely, $$\lim_{\underset{r<1}{r \to 1}}h\left(r z \right) = h^*\left(z\right), \qquad \mbox{for almost every} \ z \in \mathbb{T}. $$ 
Notice that, Theorem \ref{Th1} shows that this limit exists for all $z \in \mathbb{T}\backslash\{-1\};$ then $h^* = h$ on $\mathbb{T}\backslash\{-1\}.$ Also, one can use the uniqueness of analytic continuation of $h$ (in view of \eqref{UAC}) to justify that $h^*=h$ on $\mathbb{T}\backslash\{-1\}.$ Consequently, by utilizing the Beurling factorisation \cite[Th. 17.17]{Rud}, we show the explicit value of the integrals
$$\frac{1}{2\pi}\int_{\Re s =\frac{1}{2}} \log\left| \zeta(s) - \frac{s}{s-1}\right|\frac{|\mathrm{d}s|}{|s|^2} \quad \mbox{and} \quad \frac{\beta - \frac{1}{2}}{\pi}\int_{\Re s=\frac{1}{2}} \log\left| \zeta(s) - \frac{s}{s-1}\right|\frac{|\mathrm{d}s|}{|s-\varrho|^2},$$
where $\beta > 1/2$ is the real part of a nontrivial zero, $\varrho,$ of the Riemann zeta function. In general, we do not need to use the theory of Hardy spaces to calculate these integrals, however, it is sufficient to use a similar reasoning as in \cite{Ela} to obtain the same results. At the end of Section \ref{Section3}, we discuss how these results can provide informations about the distribution of the nontrivial zeros of the Riemann zeta function. 

The last section consists of a brief discussion of corresponding results in term of ergodic and probability theories; precisely, we provide a generalization of results obtained by Steuding in \cite{Ste}. In addition, we present a necessary condition for the truth of the Lindel\"{o}f hypothesis.

\section{The Fourier coefficients of the Riemann zeta function}
\label{Section2}

\subsection{Proof of Theorem \ref{Th1}}

Firstly, we prove the convergence in $\mathrm{L}^2$-norm of the series \eqref{Zeta12}. Since 
$$\frac{\left|\zeta\left(\frac{1}{2}+it\right) \right|^2}{\frac{1}{4}+t^2}= O\left( \frac{1}{|t|^{\frac{71}{42}-\varepsilon}}\right), \qquad \forall \varepsilon>0; \quad \mbox{as} \ |t| \to +\infty $$
then the function $t\mapsto \zeta(1/2 + it)$ belongs to $\mathrm{L_2}(\mu).$ It remains to calculate its Fourier coefficients $\ell_n$; namely,

\begin{align}
\ell_n &= \frac{1}{2\pi} \int_{- \infty}^{+\infty}\frac{\zeta\left(\frac{1}{2}+ it \right)}{\frac{1}{4}+ t^2} \left(\frac{\frac{1}{2}+it}{\frac{1}{2}- it} \right)^n \mathrm{d}t \nonumber \\ &= \frac{1}{2\pi i} \int_{\Re s = \frac{1}{2}}\frac{\zeta(s)}{s(1-s)}\left( \frac{s}{1-s}\right)^n \mathrm{d}s.
\label{Ln2}
\end{align} 
We consider the counterclockwise oriented rectangular contour $\mathcal{R}$ with vertices $1/2 - iT,$ $R - iT,$ $R+iT$ and $1/2+iT$ where $T>0$ and $R>1$ are sufficiently large. Since $\zeta(s) s^{n-1}/(1-s)^{n+1}$ is meromorphic in $\mathcal{R}$ then by Cauchy's residue theorem
$$\oint_{\mathcal{R}}\frac{\zeta(s)}{s(1-s)}\left( \frac{s}{1-s}\right)^n \mathrm{d}s = 2\pi i \mathrm{Res}_{s=1} $$
where $\mathrm{Res}_{s=1}$ is the residue of $\zeta(s) s^{n-1}/(1-s)^{n+1}$ at $s=1$ which is its only pole in the half-plane $\sigma \geq 1/2$, of order $n+2$  ($n \geq -1$). Thus,

$$\frac{1}{2\pi}\int_{-T}^{T} \frac{\zeta\left(\frac{1}{2}+ it \right)}{\frac{1}{4}+ t^2}\left(\frac{\frac{1}{2}+it}{\frac{1}{2}- it} \right)^n \mathrm{d}t = - \mathrm{Res}_{s=1} + I_1(R,T) -  I_2(R,T)+ I_2(R,-T), $$
where
$$I_1(R,T) := \frac{1}{2\pi i} \int_{R-iT}^{R+iT} \frac{\zeta(s)}{s(1-s)}\left( \frac{s}{1-s}\right)^n \mathrm{d}s ,$$
and
 
$$I_2(R,T) := \frac{1}{2\pi i}\int_{\frac{1}{2}+iT}^{R+iT} \frac{\zeta(s)}{s(1-s)}\left( \frac{s}{1-s}\right)^n \mathrm{d}s .$$
Since, for any bounded $\sigma \geq 1/2$ and a given integer $n,$
$$\frac{\zeta(s)}{s(1-s)}\left( \frac{s}{1-s}\right)^n  = O\left( \frac{1}{|t|^{\frac{154}{84} }}\right), $$
then for any bounded $R>1,$ 
$$|I_2(R,T)|=|I_2(R,-T)| = O\left( \frac{1}{T^{\frac{154}{84} }}\right), $$
as $T \to + \infty.$ Hence, 

\begin{align*}
\ell_n &= \lim_{T \to +\infty} \frac{1}{2\pi}\int_{-T}^{T} \frac{\zeta\left(\frac{1}{2}+ it \right)}{\frac{1}{4}+ t^2}\left(\frac{\frac{1}{2}+it}{\frac{1}{2}- it} \right)^n \mathrm{d}t \\
&= \begin{cases}\begin{array}{cl} - \mathrm{Res}_{s=1} & \text{if} \ n \geq -1 \\ 0 & \text{otherwise}\end{array}\end{cases} + \frac{1}{2\pi i} \int_{\Re s = R}\frac{\zeta(s)}{s(1-s)}\left( \frac{s}{1-s}\right)^n \mathrm{d}s.
\end{align*}
Now we have for a sufficiently large $R$ 
\begin{align*}
\left| \int_{\Re s = R}\frac{\zeta(s)}{s(1-s)}\left( \frac{s}{1-s}\right)^n\mathrm{d}s \right| &\leq \int_{-\infty}^{+\infty} \frac{|\zeta(R+it)|}{(R-1)^2+t^2} \left( \frac{R^2+t^2}{(R-1)^2+t^2}\right)^{\frac{n}{2}} \mathrm{d}t \\
&\leq 2\zeta(R) \int_{0}^{+\infty} \frac{\left(1 + \frac{2R-1}{(R-1)^2+t^2} \right)^{\frac{n}{2}}}{(R-1)^2+t^2}\mathrm{d}t.
\end{align*}
Since, for all $t\in \mathbb{R}$
$$ \left(1 + \frac{2R-1}{(R-1)^2+t^2} \right)^{\frac{n}{2}} \leq \begin{cases}\begin{array}{cl} 1 & \text{if} \ n \leq 0 \\ (\frac{R}{R-1})^n & \text{otherwise}\end{array}\end{cases}
 $$
then
$$\left| \int_{\Re s = R}\frac{\zeta(s)}{s(1-s)}\left( \frac{s}{1-s}\right)^n\mathrm{d}s\right| \leq \pi \frac{\zeta(R)}{R-1}\left( \frac{R}{R-1}\right)^n = O\left(\frac{1}{R} \right)$$
uniformly, as $R \to + \infty.$ Consequently, 
$$\ell_n = \begin{cases}\begin{array}{cl} - \mathrm{Res}_{s=1} & \text{if} \ n \geq -1 \\ 0 & \text{otherwise}.\end{array}\end{cases} $$

By utilizing the Laurent expansion \eqref{LauExp} we find that $\ell_{-1} = -1$, $\ell_0 = \gamma_0 - 1$ and for any integer $n \geq 1$
$$\ell_n = (-1)^n \sum_{k=1}^n \binom{n-1}{k-1} \frac{(-1)^k}{k!} \gamma_k .$$ 
Therefore, the series
$$\sum_{n \geq -1} \ell_n e_n(t) = \gamma_0 - \frac{1}{\frac{1}{2}- it} + \sum_{n \geq 1} \ell_n e_n(t)$$
converges in $\mathrm{L}^2(\mu)$ (i.e. in mean) to the function $t \mapsto\zeta(1/2 + it)$.

In order to prove the pointwise convergence, we move the line of integration in \eqref{Ln2}, for $n\geq 0$, to the line $\Re s =\sigma_1;$ where $\sigma_1$ is any fixed real number in $(0,1/2).$ Notice that, the function $\zeta(s) s^{n-1}/(1-s)^{n+1}$ is holomorphic on the strip $(\sigma_1,1/2)$ and for any $\sigma \in (\sigma_1,1/2)$ we have

$$\left| \frac{\zeta(s)}{s(1-s)} \left( \frac{s}{1-s}\right)^n \right| \leq \left| \frac{\zeta(s)}{s(1-s)}\right| = o\left( \frac{1}{|t|^{\frac{3}{2}}}\right),  $$
uniformly as $|t| \to + \infty.$ Then, for any positive integer $n$ we have

$$\ell_n = \frac{1}{2\pi i} \int_{\Re s = \sigma_1} \frac{\zeta(s)}{s(1-s)}\left( \frac{s}{1-s}\right)^n \mathrm{d}s.$$
Now, let $t_0 \in \mathbb{R}$ be fixed and $\tau =(1/2 - it_0)/(1/2+ it_0).$ We put for any positive integer $N$,

$$Z_N(\tau) = \sum_{n=0}^N \ell_n \tau^n .$$
Thus, we have
$$Z_N(\tau) =  \frac{1}{2\pi i} \int_{\Re s = \sigma_1} \frac{\zeta(s)}{s(1-s)}\sum_{n=0}^N\left( \frac{s}{1-s}\tau\right)^n \mathrm{d}s, $$
and then
\begin{equation}
 Z_N(\tau) = \frac{1}{2\pi i} \int_{\Re s = \sigma_1} \frac{\zeta(s)}{s(1-(1+\tau)s)}\left(1 - \left(\frac{s \tau}{1-s}\right)^{N+1}\right) \mathrm{d}s.
 \label{ZN}
 \end{equation}
Hence, one can use for example Lebesgue's dominated convergence theorem \cite[Th. 1.34]{Rud} to justify that

$$\lim_{N \to +\infty} Z_N(\tau) = \frac{1}{2\pi i} \int_{\Re s = \sigma_1} \frac{\zeta(s)}{s(1-(1+\tau)s)}\mathrm{d}s.$$ 
Remark that, for any $\sigma < 1/2$ and any positive integer $N,$ $|s\tau/(1-s)|^{N+1}<1$.
We deduce by Cauchy's residue theorem that
$$\sum_{n=0}^{+\infty} \ell_n e_n(t_0) = \lim_{N \to +\infty} Z_N(\tau) = \frac{1}{\tau} + \zeta\left( \frac{1}{1+\tau}\right)   = \frac{\frac{1}{2}+ it_0}{\frac{1}{2} - it_0} + \zeta\left( \frac{1}{2}+ it_0\right); $$
which completes the proof of Theorem \ref{Th1}.

\begin{remark}
It is not hard to show, by using \eqref{ZN}, that the sequences $(Z_N(\tau))_{N \geq 0}$ converges compactly on $\mathbb{R}.$ However, the series $\sum_{n=0}^{+\infty} \ell_n e_n(t)$ does not converge absolutely.
\end{remark}

\begin{remark}
Please remark that the weaker estimates arising from convexity bounds of the function 
$$\frac{\zeta(s)}{s(1-s)}\left( \frac{s}{1-s}\right)^n $$
would also suffice instead of Bourgain's result.
\end{remark}

\subsection{Proof of Theorem \ref{Th2}}

Let $\sigma_0> 1/2$ such that $\sigma_0 \neq 1.$ By a similar reasoning as in the proof of Theorem \ref{Th1}, one can show that the function $\zeta_{\sigma_0}:t \mapsto \zeta(\sigma_0 + it)$  belongs to the Hilbert space $\mathrm{L}^2(\mu)$ and, for any integer $n$, we have
\begin{align}
\langle \zeta_{\sigma_0} , e_n \rangle &=  \frac{1}{2\pi i} \int_{\Re s=\frac{1}{2}} \frac{\zeta\left( \sigma_0 - \frac{1}{2}+ s\right)}{s(1-s)}\left( \frac{s}{1-s}\right)^n \mathrm{d}s \label{LN2} \\
 &= -\sum \mathrm{Res}; \nonumber  
 \end{align}
 where $\sum \mathrm{Res}$ denotes the sum of residues at poles located in the region $\sigma > 1/2.$
  
For $1/2 < \sigma_0 < 1;$ a short comuptation shows that the integrand has a simple pole at $s=3/2 - \sigma_0$ with residue $(3/2-\sigma_0)^{n-1}/(\sigma_0 - 1/2)^{n+1}$  ($n \in \mathbb{Z}$) and has poles of order $n+1$ ($n\geq 0$) at $s=1$ with residue

$$\begin{cases}\begin{array}{cll} -\zeta(\sigma_0 +1/2)  &\text{for} & n = 0,  \\ \\ \displaystyle (-1)^{n-1}\sum_{k=1}^n\binom{n-1}{k-1}\frac{1}{k!}\zeta^{(k)}\left(\sigma_0+\frac{1}{2}\right)  & \text{for} & n\geq 1 .\end{array}\end{cases}$$
Thus,

$$\langle \zeta_{\sigma_0} , e_n \rangle = \begin{cases}\begin{array}{cll}  \frac{(-1)^n}{\left(\sigma_0 - \frac{1}{2}\right)^2}\left( \frac{  \sigma_0-\frac{3}{2}}{\sigma_0 - \frac{1}{2}}\right)^{n-1} &\text{for} & n \leq -1,  \\ \ell_0(\sigma_0) & \text{for} & n=0,  \\ \ell_n(\sigma_0) & \text{for} & n\geq 1, \end{array}\end{cases}$$
where
$$\ell_0(\sigma_0)=\zeta\left( \sigma_0 + \frac{1}{2} \right) - \frac{1}{\sigma_0 - \frac{1}{2}} - \frac{1}{\frac{3}{2} - \sigma_0}$$
and
\begin{align*}
\ell_n(\sigma_0) &= (-1)^{n}\sum_{k=1}^n\binom{n-1}{k-1}\frac{1}{k!}\zeta^{(k)}\left(\sigma_0+\frac{1}{2}\right)  - \frac{(-1)^n}{\left(\sigma_0 - \frac{1}{2}\right)^2}\left( \frac{  \sigma_0-\frac{3}{2}}{\sigma_0 - \frac{1}{2}}\right)^{n-1} \\ &= (-1)^{n}\sum_{k=1}^n\binom{n-1}{k-1}\left( \frac{1}{k!}\zeta^{(k)}\left(\sigma_0+\frac{1}{2}\right) - \frac{(-1)^{k}}{\left( \sigma_0 - \frac{1}{2} \right)^{k+1}}  \right).
\end{align*}
Hence, the following equality holds in $\mathrm{L}^2$-norm, for any fixed $\sigma_0 \in (1/2,1)$,
\begin{align*}
\zeta(\sigma_0+it) &=  \sum_{n \in \mathbb{Z}} \langle \zeta_{\sigma_0} , e_n \rangle e_n(t) \\
&=\sum_{n \leq -1} \langle \zeta_{\sigma_0} , e_n \rangle e_n(t) + \langle \zeta_{\sigma_0} , e_0 \rangle + \sum_{n \geq 1} \ell_n(\sigma_0) e_n(t);
\end{align*}
where
\begin{align*}
\sum_{n \leq -1} \langle \zeta_{\sigma_0} , e_n \rangle e_n(t) &= \frac{-1}{\left( \frac{3}{2} - \sigma_0 \right)\left(\sigma_0 - \frac{1}{2} \right)} \sum_{n\geq 1} \left( \frac{\sigma_0-\frac{1}{2}}{\frac{3}{2}-\sigma_0}\frac{\frac{1}{2}+it}{\frac{1}{2}-it}\right)^n \\ &= -\frac{\frac{1}{2}+it}{(1-\sigma_0-it)\left(\frac{3}{2} - \sigma_0\right)} =  \frac{1}{\sigma_0-1+it}-\frac{1}{\sigma_0-\frac{3}{2}}.
\end{align*}
Notice that, the series above converges absolutely for any $\sigma_0 \in (1/2,1)$ and any $t\in \mathbb{R}.$ Therefore, we obtain the first formula in Theorem \ref{Th2}. Moreover, one can use a similar reasoning as in the proof of Theorem \ref{Th1}, applied to \eqref{LN2} on moving the line of integration to the left of the critical line, to prove the pointwise convergence of the sequences $(\sum_{n=0}^N \ell_n(\sigma_0) e_n(t)).$

For the case $\sigma_0>1$, notice that the integrand has (only) pole of order $n+1$ ($n\geq 0$) at $s=1.$ A similar reasoning, as in above, yields a complete proof of Theorem \ref{Th2}.    

It should be noted that the function $t \mapsto \zeta(1+it)$ does not belong to $\mathrm{L}^2(\mu)$, since $|\zeta(1+it)|^2\sim t^{-2}$ as $t\to 0.$ Moreover, we notice in view of \eqref{LauExp} that 
$$\lim_{\sigma_0 \to \frac{1}{2}} \langle \zeta_{\sigma_0} , e_n \rangle  = \ell_n,  $$
for any given integer $n$. Hence, Theorem \ref{Th2} can be considered as a generalization of Theorem \ref{Th1}. Moreover, we deduce the following result.

\begin{corollary}
For any real numbers $a$ and $b$ in $[1/2,1)$, we have
$$\int_{\mathbb{R}}\zeta(a+ it)\zeta(b+it) \mathrm{d}\mu(t) = F(a,b) + F(b,a) + \ell_0(a)\ell_0(b);$$
where,
$$F(a,b):= -\frac{1}{(a-\frac{1}{2})^2}\sum_{n \geq 1}\ell_n(b) \left( \frac{a - \frac12}{ \frac{3}{2}- a} \right)^{n+1}, $$
and with $\ell_n(1/2) = \ell_n,$ for any integer $n \geq 0.$  
\label{Core} \end{corollary}
\begin{proof}
Since for any fixed $\sigma \in [1/2,1),$
$$ \frac{\left|\zeta(\sigma+it)\right|^2}{\frac{1}{4}+t^2} = O\left( \frac{1}{|t|^{\frac{35}{21}}}\right),\qquad \mbox{as} \quad |t| \to +\infty. $$
Then $\zeta_{\sigma}$ belongs to $\mathrm{L}^2(\mu).$ Therefore, by Theorem \ref{Th2}, for any $a,b \in [1/2,1)$ we have
\begin{align*}
\int_{\mathbb{R}}\zeta(a+it)\zeta(b+it) \mathrm{d}\mu &= \langle \zeta_{a} , \overline{\zeta_b} \rangle \\
&= \sum_{n \in \mathbb{Z}}\sum_{m \in \mathbb{Z}} \langle \zeta_{a} , e_n \rangle  \langle \zeta_{b} , e_{m} \rangle \langle e_n , \overline{e_m} \rangle  .
\end{align*}
Since,
$$  \langle e_n , \overline{e_m} \rangle =  \langle e_n , e_{-m} \rangle =  \begin{cases}\begin{array}{lc} 1 & \text{if} \ n=-m \\ 0 & \text{otherwise}\end{array}\end{cases}$$
then
\begin{align*}
 \langle \zeta_{a} , \overline{\zeta_b} \rangle &= \sum_{n \in \mathbb{Z}} \langle \zeta_{a} , e_n \rangle  \langle \zeta_{b} , e_{-n} \rangle \\
 &= \ell_0(a)\ell_0(b) + F(a,b) + F(b,a);
 \end{align*}
 which completes the proof.
\end{proof}

The well-known Parseval's identity in the Hilbert space $\mathrm{L}^2(\mu)$ asserts that 
$$\left\|f \right\|_2^2:= \frac{1}{2\pi}\int_{\mathbb{R}}\left| f(t) \right|^2 \mathrm{d}\mu(t) = \sum_{n \in \mathbb{Z}} |\langle f , e_n \rangle|^2, \qquad  (f \in \mathrm{L}^2(\mu)).$$
Thus, by applying that onto the function $t \mapsto \zeta(1/2 + it)$, we obtain
\begin{corollary}

$$\sum_{n \geq 0} \ell_n^2 = \log(2\pi)- \gamma_0 - 1 \approx  0.2606614$$
\label{ThPI}\end{corollary}
\begin{proof}
The proof of Corollary \ref{ThPI} is based on a beautiful result due to Coffey \cite{Cof}, namely
\begin{equation}
\frac{1}{2\pi} \int_{\mathbb{R}} \left| \zeta\left( \frac{1}{2} + it \right) \right|^2 \mathrm{d}\mu(t) = \log(2\pi) - \gamma_0 .
\label{Coff}
\end{equation}
\end{proof}

We would like to close this section with the following remark;
\begin{remark}
Let $\mathfrak{B}$ be the binomial transform defined, for a given sequence $(a_n)_{n \geq 0}$, by
 $$b_n = \mathfrak{B}(a_n) := \sum_{k=0}^n\binom{n}{k}(-1)^{n-k}a_k. $$
Thus the inverse of $\mathfrak{B}$ is 
$$a_n = \mathfrak{B}^{-1}(b_n) = \sum_{k=0}^n\binom{n}{k} b_k. $$
Hence, we deduce that
$$\frac{\gamma_n}{n!} = \sum_{k=1}^n \binom{n-1}{k-1} \ell_k. $$
To show that, it is sufficient te rewrite $\ell_n$ as follows;

$$\ell_n = \sum_{k=1}^n\binom{n-1}{k-1}(-1)^{n-k}\frac{\gamma_k}{k!} = \frac{1}{n} \sum_{k=0}^{n}\binom{n}{k}(-1)^{n-k}\frac{k\gamma_k}{k!} = \frac{1}{n} \mathfrak{B}\left( \frac{n \gamma_n}{n!} \right).$$
\label{RMK}

\end{remark}
\section{The generating function for $(\ell_n)_{n\geq 0}$ and the nontrivial zeros of the Riemann zeta function}
\label{Section3}

In this section we show that the value-distribution of the Riemann zeta function on the half-plane $\sigma \geq 1/2$ is closely dependent on the Fourier coefficients $\ell_n$.

\subsection{Series representation of the Riemann zeta function} 
Let us start with the main result of this section.

\begin{theorem}
For any complex number $s$ in the half-plane $\sigma > \frac{1}{2}$, we have
$$\zeta(s) = \frac{s}{s-1} + \sum_{n \geq 0} \ell_n \left( \frac{1-s}{s}\right)^n $$
\label{ThG}\end{theorem}
\begin{proof}
Let $z$ be a complex number in the open unit disk, $\mathbb{D}$. Since $(\ell_n)_{n \geq 0}$ is a bounded sequence then the series $\sum_{n \geq 0} \ell_n z^n$ is absolutely convergent on $\mathbb{D}$. Hence, 
\begin{align*}
\sum_{n \geq 0} \ell_n z^n &=  \frac{1}{2\pi i}\int_{\Re s = \frac{1}{2}} \frac{\zeta(s)}{s(1-s)} \sum_{n \geq 0} \left( \frac{s}{1-s} z\right)^n \mathrm{d}s \\ &= \frac{1}{2\pi i}\int_{\Re s = \frac{1}{2}} \frac{\zeta(s)}{s(1-(1+z)s)}\mathrm{d}s 
\end{align*}
Thus, by Cauchy's residue theorem, as in the proof of Theorem \ref{Th1}, and since $\Re[ 1/(1+z)] > 1/2$, for any $z \in \mathbb{D}$. We obtain
\begin{equation}
\sum_{n \geq 0} \ell_n z^n = \frac{1}{z} + \zeta\left( \frac{1}{1+z}\right).  
\label{Hardy}
\end{equation}
The M\"{o}bius transformation $z\mapsto 1/(1+z)$, also called Cayley map, maps the unit disk to the half-plane $\{ s \in \mathbb{C}, \ \Re s > 1/2\}$. Therefore, the substitution $s = 1/(1+z)$ completes the proof of Theorem \ref{ThG}.
\end{proof}

We deduce the following explicit values
\begin{corollary}
For any $\sigma \in [1/2,1)$ we have

$$\int_{\mathbb{R}}\zeta(\sigma +it)\zeta\left( \frac{1}{2}+it\right)\mathrm{d}\mu(t) = (\gamma_0-1)\zeta\left( \sigma+ \frac{1}{2}\right) -  \frac{\zeta\left(\frac{3}{2} - \sigma \right)}{(\sigma-\frac{1}{2})(\frac{3}{2}-\sigma)} - \frac{1}{(\sigma-\frac{1}{2})^2}.  $$
The case of $\sigma=1/2$ is included as a limit of the term in the right-hand side, by using \eqref{LauExp}, as $\sigma \to 1/2.$
\label{wow}
\end{corollary}
\begin{proof}
Let $a=\sigma$ and $b=1/2$ in Corollary \ref{Core}, then we obtain
$$\int_{\mathbb{R}}\zeta(\sigma +it)\zeta\left( \frac{1}{2}+it\right)\mathrm{d}\mu(t) = F\left( \sigma, \frac{1}{2} \right) + (\gamma_0-1) \ell_0(\sigma).$$
Thus, by Theorem \ref{ThG},
\begin{align*}
F\left( \sigma, \frac{1}{2} \right) &= \frac{-1}{(\sigma-\frac{1}{2})(\frac{3}{2}-\sigma)} \sum_{n \geq 1} \ell_n \left( \frac{1-(\frac{3}{2}-\sigma)}{ \frac{3}{2}-\sigma} \right)^n \\ &=  \frac{-1}{(\sigma-\frac{1}{2})(\frac{3}{2}-\sigma)}\left( \zeta\left(\frac{3}{2}-\sigma \right)- \frac{1}{\frac{1}{2}-\sigma}- \gamma_0\right),
\end{align*}
and the substitution $\ell_0(\sigma) = \zeta(\sigma + 1/2) - 1/((\sigma-1/2)(3/2-\sigma))$ completes the proof.
\end{proof}

Furthermore, the formula \eqref{Hardy} allows us to state the following result
\begin{corollary}
The Riemann hypothesis is true if and only if 
$$\left\{ z \in \mathbb{D}, \quad h(z) = \frac{1}{z}  \right\} = \emptyset ,$$
where $h$ is the holomorphic function defined on $\mathbb{D}$ by
$$h(z) := \sum_{n \geq 0} \ell_n z^n.$$ 
\label{RHC}\end{corollary}

So, it seems to be important to study the distribution of values of the function $h$. By comparing Theorem \ref{ThG} with the well-know formula of the Riemann zeta function,
$$\zeta(s) = \frac{s}{s-1} - s\int_1^{+\infty}\frac{\{x\}}{x^{s+1}}\mathrm{d}x$$
which is valid for all $\sigma>0$, we deduce that 

\begin{equation}
\sum_{n\geq 0}\ell_n \left( \frac{1-s}{s}\right)^n = -s \int_1^{+\infty}\frac{\{x\}}{x^{s+1}}\mathrm{d}x = -s \varphi(s), \qquad \sigma >\frac{1}{2},
\label{phif}
\end{equation}
$\{\cdot\}$ denotes the fractional part function. Moreover, by using the Cauchy-Schwarz inequality and  Corollary \ref{ThPI} we obtain the following upper bound,

\begin{corollary}
For all $\sigma > 1/2$ we have
$$\left|\zeta(s) - \frac{s}{s-1} \right| \leq \sqrt{\frac{\log(2\pi) - \gamma_0 -1}{2\sigma- 1}}|s| .$$
\end{corollary}  
\begin{proof}
 Let $\sigma > 1/2,$ then by using, respectively, triangle and Cauchy-Schwarz inequalities we obtain
\begin{align*}
\left|\zeta(s) - \frac{s}{s-1} \right| &\leq \sum_{n \geq 0} |\ell_n| \left| \frac{1-s}{s}\right|^n \\ &\leq  \sqrt{\sum_{n \geq 0} |\ell_n|^2} \sqrt{\sum_{n \geq 0} \left| \frac{1-s}{s}\right|^{2n}} \\ &= \sqrt{\frac{\log(2\pi) - \gamma_0 -1}{\left( 1 - \left| \frac{1-s}{s}\right|^2\right)}}.
\end{align*}  
Remark that $|(1-s)/s|< 1,$ for all $\sigma > 1/2.$ Since 
$$ 1 - \left| \frac{1-s}{s}\right|^2 = \frac{2\sigma-1}{|s|^2}, $$
then, for all $\sigma > 1/2$
$$\left|\zeta(s) - \frac{s}{s-1} \right| \leq \sqrt{\frac{\log(2\pi) - \gamma_0 -1}{2\sigma- 1}}|s|.$$
\end{proof}

\subsection{An approach to hypothetical zeros of the Riemann zeta function}
Let $\varrho_0$ ($\Re \varrho_0>1/2$) be an hypothetical zero of $\zeta(s);$ we put $z_0 = (1-\varrho_0)/\varrho_0.$ Then, by Theorem \ref{RHC},
$$ f(z_0) := -1 + \sum_{n \geq 0} \ell_n z_0^{n+1} = 0 .$$
Please, notice that $|z_0|<1.$\\
Now, we consider for a sufficiently large positive integer $N$, the following partial sums
$$f_N(z) = - 1 + \sum_{n=0}^N \ell_n z^{n+1}. $$
Then, by using the Cauchy-Schwarz inequality we obtain, for any complex number $|z|<1,$ 
\begin{align*}
|f(z) - f_N(z)| &=\left| \sum_{n=N+1}^{+\infty} \ell_n z^{n+1} \right| \\
&\leq |z| \left(\sum_{n=N+1}^{+\infty} \ell_n^2 \right)^{\frac{1}{2}} \left(\sum_{n=N+1}^{+\infty} |z|^{2n} \right)^{\frac{1}{2}} \\ 
&=   \left(\sum_{n=N+1}^{+\infty} \ell_n^2 \right)^{\frac{1}{2}} \frac{|z|^{N+2}}{\sqrt{1-|z|}}.
\end{align*}
Thus, for $z=z_0$ we have
$$|f_N(z_0)| \leq \left(\sum_{n=N+1}^{+\infty} \ell_n^2 \right)^{\frac{1}{2}} \frac{|z_0|^{N+2}}{\sqrt{1-|z_0|}} = o\left(|z_0|^{N+2} \right) \quad \mbox{as} \ N \to + \infty .$$
Therefore, it would be an interesting problem to make a numerical search for the zeros of the polynomials $f_N$ in the unit disk using mathematical software. Note that tables of values of $\gamma_k$ are available and so for the $\ell_n$ as well by using the transformation formula indicated in Remark \ref{RMK}.

We should not forget to mention an interesting related result of Jentzsch \cite{Jen} which states that, every point of the circle of convergence of a power series is a  limit point of its partial sums. 
   
\subsection{Beurling factorization of $\zeta(s) - s/(s-1)$}

Let $h$ be the function defined in Corollary \ref{RHC}. Since the power series $\sum_{n \geq 0} \ell_n z^n$ is absolutely convergent, for all $z\in\mathbb{D},$ then $h$ is holomorphic on $\mathbb{D}$. Moreover, it follows from Corollary \ref{ThPI} and \cite[Th. 17.12]{Rud}, that the function $h$ belongs to the Hardy space $\mathrm{H}^2(\mathbb{D})$ and
$$\|h\|_{\mathrm{H}^2}^2 = \log(2\pi) - \gamma_0 - 1.$$
Therefore, by \cite[Th. 17.11]{Rud}, the nontangential limits, $h^*,$ exist almost everywhere on the unit circle $\mathbb{T}=\partial \mathbb{D}.$ However, it follows from Theorem \ref{Th1} (or from the analytic continuation of $h$, in view of \eqref{UAC}) that $h^*=h$ on $\mathbb{T}\backslash\{-1\};$ namely,
$$\lim_{\underset{r<1}{r \to 1}}h\left( rz\right) = h(z)= \frac{1}{z} + \zeta\left( \frac{1}{1+z} \right), \qquad \forall \ z\in \mathbb{T}\backslash\{-1\}.$$   
Let $\Omega$ denote the set of zeros of the function $h$ located in $\mathbb{D}.$ Notice that, there is no zeros of the Riemann zeta function in $\Omega .$ 
Hence, we obtain the following important result.
\begin{theorem}
The following representation is valid for $\sigma > 1/2,$

$$\zeta(s) - \frac{s}{s-1} =Q(s) \prod_{w \in \Omega}\left(\frac{|w|}{w} \frac{s(1+w)-1}{s(1+\overline{w}) - \overline{w}}\right); $$
where
$$Q(s) = \exp\left( \frac{1}{2\pi} \int_{-\pi}^{\pi}\frac{(e^{i\theta} -1)s+1}{(e^{i\theta} +1)s-1}\log\left|h(e^{i\theta})\right| \mathrm{d}\theta\right). $$
\label{ThB}\end{theorem}
\begin{proof}
Since $h \in \mathrm{H}^2(\mathbb{D}),$ then by \cite[Th. 17.17]{Rud} we have $\log|h| \in \mathrm{L}^1(\mathbb{T})$, the outer function
$$Q_h(z) = \exp\left( \frac{1}{2\pi} \int_{-\pi}^{\pi}\frac{e^{i\theta}+z }{e^{i\theta} - z }\log\left|h(e^{i\theta})\right| \mathrm{d}\theta\right) $$ 
is in $\mathrm{H}^2(\mathbb{D})$ and $h$ has the Beurling factorization $h=M_hQ_h=cB_hSQ_h;$ where $c=e^{i \Im h(0)}=1,$ $B_h$ is the Blaschke product formed with the zeros of $h$,
$$B_h(z) = \prod_{w \in \Omega}\left( \frac{|w|}{w}\frac{w-z}{1 - \overline{w}z} \right), \qquad (z \in \mathbb{D}) $$ 
and $S$ is the singular function defined by
$$S(z) = \exp\left( -\int_{\mathbb{T}} \frac{\lambda +z}{\lambda-z} \mathrm{d}\nu(\lambda)\right); $$
where $\nu$ is a finite positive Borel measure on $\mathbb{T}$ which is singular with respect to the Lebesgue measure $\mathrm{d}\theta .$ Since $h$ is analytic across $\mathbb{T}\backslash\{-1\}$ then, by \cite[Th. 6.3]{Gar}, $S$ is analytic across $\mathbb{T}\backslash\{-1\}$. That means that, $\nu(\mathbb{T}\backslash\{-1\}) = 0$ and $\nu(\{-1\}) = \tau \geq 0$. Hence, $S$ has the following form
$$S(z) = \exp\left( \tau \frac{z-1}{z+1} \right). $$
Now, if $\tau \neq 0$ then it would follow that $\lim_{z \nearrow -1} S(z)/(1+z) = 0;$ however, we have
\begin{align*}
h(z) &= \frac{1}{z} + \zeta\left( \frac{1}{1+z}\right) \\ 
&= \zeta(s) -1 - \frac{1}{s-1}\qquad \qquad  \quad \ \left(z+1=\frac{1}{s}\right) \\
&= - \frac{1}{s-1}+ \frac{1}{2^s} + O\left( \frac{1}{3^{\sigma}}\right), \qquad \mbox{as} \ \sigma \to +\infty
\end{align*} 
then $\lim_{z \nearrow -1}h(z)/(1+z) = -1 \neq 0.$ We conclude that $\tau = 0.$ The substitution $z = (1-s)/s$ (for $\sigma > 1/2$) completes the proof of Theorem \ref{ThB}. 
\end{proof}

Consequently, We obtain the following logarithmic integrals,

\begin{corollary}
We have,
$$ \frac{1}{2\pi}\int_{\Re s = \frac{1}{2}}\log\left|\zeta(s) - \frac{s}{s-1} \right| \frac{|\mathrm{d}s|}{|s|^2} = \log(1-\gamma_0) + \sum_{w \in \Omega}\log\left| \frac{1}{w} \right|;$$
and if $\varrho = \beta + i \gamma$ with $\beta > 1/2$ is a nontrivial zero of the Riemann zeta function, then

\begin{equation}
\log\left| \frac{\varrho}{1-\varrho}\right| = \frac{2\beta - 1}{2\pi} \int_{\Re s = \frac{1}{2}}\log\left|\zeta(s) - \frac{s}{s-1} \right| \frac{|\mathrm{d}s|}{|\varrho - s|^2} + \sum_{w \in \Omega}\log\left|\frac{\varrho(1+w)-1}{\varrho(1+\overline{w}) - \overline{w}}\right|. 
\label{secondrho}
\end{equation}
\label{Clog}\end{corollary} 
\begin{proof}
Let $\Re u>1/2.$ Since,
\begin{align*}
\log|Q(u)|& = \Re\left( \frac{1}{2\pi} \int_{-\pi}^{\pi}\frac{(e^{i\theta} -1)u+1}{(e^{i\theta} +1)u-1}\log\left|h(e^{i\theta})\right| \mathrm{d}\theta\right) \\
&= \frac{2\Re u - 1}{2\pi} \int_{\Re u = \frac{1}{2}}\log\left|\zeta(s) - \frac{s}{s-1} \right|\frac{|\mathrm{d}s|}{|s - u|^2},
\end{align*}
 then by Theorem \ref{ThB},
$$\log\left|\zeta(u) - \frac{u}{u-1} \right| = \frac{2\Re u - 1}{2\pi} \int_{\Re u = \frac{1}{2}}\log\left|\zeta(s) - \frac{s}{s-1} \right|\frac{|\mathrm{d}s|}{|s - u|^2} + \sum_{w \in \Omega} \log\left|\frac{u(1+w)-1}{u(1+\overline{w}) - \overline{w}}\right|. $$
Hence, the first equation holds by letting $u=1$ and the second equation holds by letting $u=\varrho.$ 
\end{proof}

Let $\varphi$ be the function defined in \eqref{phif},
$$ \varphi(s) = \int_1^{+\infty} \frac{\{x\}}{x^{s+1}}\mathrm{d}x, \qquad \sigma \geq \frac{1}{2} . $$
Then, it is clear that $h(z)$ and $\varphi(1/(1+z))$ have the same zeros. Moreover, since $\varphi(1/(1+z))$ is bounded on $\mathbb{D}$ (i.e. $\varphi(1/(1+z)) \in \mathrm{H}^{\infty}(\mathbb{D})$) then the series $\sum_{w \in \Omega}(1-|w|)$ is absolutely convergent, see \cite[15.22]{Rud}. That means that, if $h(z)$ has infinitely many zeros then almost of these zeros are near to the boundary $\mathbb{T}.$ Precisely, we have,

\begin{corollary}
 If $w \in \mathbb{D}$ is a zero of $h$, then 
 
 $$ |w|\geq 1 - \log\left( \frac{\sqrt{\log(2\pi) - \gamma_0 - 1}}{1-\gamma_0}\right) \approx 0.8113 .$$
\label{w}\end{corollary}
\begin{proof}
The first formula in Corollary \ref{Clog}, can be rewritten as

$$\sum_{w \in \Omega}\log\left| \frac{1}{w}\right| = - \log(1-\gamma_0) + \int_{\mathbb{R}}\log\left| h\left( \frac{\frac{1}{2}-it}{\frac{1}{2}+it}\right)\right| \mathrm{d}\mu(t).$$
Thus by using the Jensen inequality, we obtain

$$\int_{\mathbb{R}}\log\left| h\left( \frac{\frac{1}{2}-it}{\frac{1}{2}+it}\right)\right| \mathrm{d}\mu(t) \leq \log\left( \|h\|_2\right). $$
Hence,
$$ \sum_{w \in \Omega}\log\left| \frac{1}{w}\right| \leq \log\left( \frac{\|h\|_2}{1-\gamma_0} \right). $$
Now, since
$$1-|w| \leq \sum_{n \geq 1} \frac{(1-|w|)^n}{n} =  \log\left| \frac{1}{w}\right| \leq  \sum_{w \in \Omega}\log\left| \frac{1}{w}\right|, $$
then
$$1-|w| \leq \log\left( \frac{\|h\|_2}{1-\gamma_0} \right). $$
Notice that the value of $\|h\|_2^2$ is given in Corollary \ref{ThPI}; the proof is complete.  
\end{proof}
\begin{remark}

One can use directly in the proof of Corollary above the following lower bound,
$$ |w|  \geq \frac{1-\gamma_0}{\|h\|_2} \approx 0.828.  $$

Actually, it is natural that the function $\varphi$ takes a small values, in mean, on the critical line. In fact, since $h((1-s)/s) = s\varphi(s)$ then
$$\int_0^{+\infty}\left| \varphi\left(\frac{1}{2}+ it \right)\right|^2 \mathrm{d}t = \pi \|h \|_2^2 = \pi\left(\log(2\pi) - \gamma_0 -1\right) \approx 0.8189 . $$
\end{remark}

It should be noted that the distribution of zeros of the function $\varphi$ is very important since relevant information on the distribution of values of the Riemann zeta function. Indeed, if we know the distribution of zeros of the function $\varphi$ then we could find an optimal bound of the series in \eqref{secondrho} which is dependent on the term, $\log|\varrho/(1-\varrho)|$, appeared in \cite{Bal} and \cite{Ela}; namely

\begin{equation}
\int_{\mathbb{R}}\log\left|\zeta\left( \frac{1}{2}+it\right) \right|\mathrm{d}\mu(t) = \sum_{\beta > \frac{1}{2}}\log \left| \frac{\varrho}{1-\varrho}\right|.
\label{Houcine}
\end{equation}
Notice that, the Riemann hypothesis is equivalent to the fact that $|\varrho/(1-\varrho)|=1.$ However, a similar reasoning as in the proof of Corollary \ref{w} yields,

\begin{corollary}
If $\varrho$ is a nontrivial zero of the Riemann zeta function located on the half-plane $\sigma \geq 1/2$, then
$$1\geq \left| \frac{1-\varrho}{\varrho} \right| \geq \left(\|\zeta_{\frac{1}{2}}\|_2\right)^{-1} \approx 0.8906.  $$ 
\end{corollary}
\begin{proof}
It is sufficient to apply the Jensen inequality to \eqref{Houcine}. Notice that $$\|\zeta_{\frac{1}{2}}\|_2 = \sqrt{\log(2\pi) - \gamma_0}. $$
\end{proof}

\section{Ergodic interpretation and the Lindel\"{o}f hypothesis}

There is a slightly different interpretations of our results. For instance, in terms of Brownian motion, if $B_t$ denotes the complex Brownian motion starting at the origin and $\tau$ is the passage time to the critical line, then the imaginary part of $B_{\tau}$  has Cauchy distribution with scale $1/2$. Hence, one can calculate the expectation of $\zeta(B_{\tau})f(B_{\tau})$, for some complex-valued functions $f \in \mathrm{L}^2(\mu)$, as a real series involving the sequence $(\ell_n)_{n\geq -1}$; thanks to Theorem \ref{Th1}.

\subsection{Interpretation in terms of ergodic theory}
Let $T$ be the transformation defined on $\mathbb{R}$ by

$$Tx := \begin{cases}\begin{array}{cl}
\frac{1}{2}\left(x - \frac{1}{4x} \right) & \text{if} \ x \in \mathbb{R}\backslash\{0\} \\ 0 & \text{otherwise} .
\end{array} \end{cases}$$
$T$ is so-called a Boole-transformation. In \cite[S. 3]{Ela}, the authors showed that $T$ is measure preserving with restpect to the probability measure $\mu$ defined in \eqref{cauchy} and, then, $(\mathbb{R},\mathcal{B}, \mu, T)$ is an ergodic system. Hence, by the pointwise ergodic theorem (or Birkhoff-khinchin theorem, see \cite{Bir} and \cite{Khin}), we obtain the following result,

\begin{theorem}
For any $g \in \mathrm{L}^2(\mu)$ and almost all real $x$, we have

$$\lim_{N\to \infty} \frac{1}{N} \sum_{n=0}^{N-1} \zeta\left( \frac{1}{2} + i T^n x\right)g\left(T^n x \right) = - a_{1} + \sum_{m \geq 0} \ell_m a_{-m},$$
where $(a_m)_{m\in \mathbb{Z}}$ are the Fourier coefficients of $g$ and $T^nx = T\circ T^{n-1} x $, with $T^0x=x$.
\label{ErgoTh}
\end{theorem}
\begin{proof}
Actually, the pointwise ergodic theorem is valid for the functions in $\mathrm{L}^1$-space. However, since $\mu$ is positive and finite then $\mathrm{L}^2(\mu) \subset \mathrm{L}^1(\mu).$ Moreover, by Cauchy-Shwarz inequality, if $f,g \in \mathrm{L}^2(\mu)$ then $f\overline{g} \in \mathrm{L}^1(\mu).$\\
Now, since the system $(\mathbb{R},\mathcal{B}, \mu, T)$ is ergodic and $\zeta_{1/2} \in \mathrm{L}^2(\mu)$ then for any complex-valued function $g \in \mathrm{L}^2(\mu)$ we have $fg \in \mathrm{L}^1(\mu).$ By the pointwise ergodic theorem, the limit of the Ces\`aro mean 
$$\frac{1}{N} \sum_{n=0}^{N-1} \zeta\left( \frac{1}{2} + i T^n x\right)g\left(T^n x \right) $$
exists for almost all $x \in \mathbb{R}$ and 
 $$\lim_{N\to \infty} \frac{1}{N} \sum_{n=0}^{N-1} \zeta\left( \frac{1}{2} + i T^n x\right)g\left(T^n x \right) = \int_{\mathbb{R}}\zeta\left( \frac{1}{2}+it \right)g(t) \mathrm{d}\mu(t)$$
 almost everywhere. Hence, if
 $$g(t) = \sum_{m \in \mathbb{Z}}a_m e_m(t), \qquad \mbox{in} \ \mathrm{L}^2(\mu) $$
 then by Theorem \ref{Th1}, we have
 $$\int_{\mathbb{R}}\zeta\left( \frac{1}{2}+it \right)g(t) \mathrm{d}\mu(t) = \langle \zeta_{\frac{1}{2}}, \overline{g}\rangle = \sum_{m \geq -1}\ell_m a_{-m},$$
 which completes the proof.
\end{proof}

\begin{remark}
Please remark that, one can generalize Theorem \ref{ErgoTh} by utilizing Theorem \ref{Th2} instead of Theorem \ref{Th1}.
\end{remark}

We should not forget to mention that Steuding \cite{Ste} proved an analogue of Theorem \ref{ErgoTh} for $g \equiv 1.$ Moreover, the case of $g \equiv \zeta_a$, $a \in [1/2,1],$ is proven in Corollary \ref{wow}.   

\subsection{On the Lindel\"of hypothesis}

 The Lindel\"of hypothesis is the most interesting open question in the value-distribution of the Riemann zeta function topic. In \cite{Lind}, Lindel\"{o}f  expressed his belief that $\zeta(s)$ is bounded on any vertical line in the strip $1/2< \sigma <1$, which would imply the Lindel\"{o}f hypothesis. However, Lindel\"of's boundedness conjecture is false as Corollary \cite[2 p.184]{Edw} shows. Therefore, we deduce the following result,

\begin{theorem}
 For any $\sigma_0 \in [1/2,1),$ we have
$$\sum_{n \geq 1} \left| \ell_n(\sigma_0) \right| = + \infty .$$
\end{theorem} 
\label{Ddiv}\begin{proof}
Let $\sigma_0 \in [1/2,1).$ If the series $\sum_{n \geq 0} \ell_n(\sigma_0)$ is absolutely convergent, then by Theorem \ref{Th2}

$$|\zeta(\sigma_0+it)| \leq \sum_{n \geq 1}\left| \ell_n(\sigma_0) \right| + \left|\zeta\left(\sigma_0 +\frac{1}{2} \right) - \frac{1}{\sigma_0 - \frac{1}{2}} \right| + \frac{1}{\sqrt{(1-\sigma_0)^2 + t^2}},$$
i.e. the Riemann zeta function is bounded on the vertical line $\Re s = \sigma_0;$   which contradicts with Corollary \cite[2 p.184]{Edw}. 
\end{proof} 
 
Actually, Theorem \ref{ThG} shows that the behaviour of the Riemann zeta function in the half-plane $\sigma>1/2$ is strongly related to the growth rate of the sequence $(\ell_n)_{n \geq 0}$, as $n \to + \infty$. For the sake of completeness, let us state the following result

\begin{corollary}
For any fixed $\sigma > 1/2,$
$$\zeta(s)  = o\left( \mathrm{Li}_{\frac{1}{2}}\left(\sqrt{1 - \frac{2\sigma - 1}{|s|^2}}\right)\right), \quad \mbox{as} \ |t| \to +\infty ,$$
where $\mathrm{Li}_\alpha(z)$ denotes the polylogarithm function, defined by
$$\mathrm{Li}_\alpha(z) := \sum_{n \geq 1}\frac{z^n}{n^\alpha}, \qquad z \in \mathbb{D}. $$
\end{corollary}
 \begin{proof}
 Firstly, we prove that $$ \ell_n = o\left( \frac{1}{\sqrt{n}}\right), \quad \mbox{as} \ n \to +\infty .$$
 In fact, if there exist $\varepsilon > 0$ such that for a sufficiently large $n > n_0 \in \mathbb{N}$ we have $\sqrt{n}|\ell_n| > \varepsilon$ then 
 $$\sum_{n> n_0} \ell_n^2 > \varepsilon \sum_{n > n_0 } \frac{1}{n} = + \infty , $$
 which contradicts with Corollary \ref{ThPI}. Hence, for any $\varepsilon > 0$ there exists $n_0\in \mathbb{N}$ such that for any $n > n_0$ we have $\sqrt{n}|\ell_n| <\varepsilon.$ Namely, $\ell_n = o(1/\sqrt{n}).$

Now let $\sigma > 1/2,$ then by using Theorem \ref{ThG},
\begin{align*}
|\zeta(s)| &\leq \left|\frac{s}{s-1}\right| + \sum_{n=0}^{n_0}|\ell_n| \left|\frac{1-s}{s} \right|^n + \sum_{n>n_0}|\ell_n| \left|\frac{1-s}{s} \right|^n \\
&= \left|\frac{s}{s-1}\right| + \sum_{n=0}^{n_0}|\ell_n| \left|\frac{1-s}{s} \right|^n + o\left( \sum_{n>n_0} \frac{\left|\frac{1-s}{s} \right|^n}{\sqrt{n}} \right) \\ &= O(1) + o\left( \mathrm{Li}_{\frac{1}{2}}\left(\sqrt{1 - \frac{2\sigma - 1}{|s|^2}}\right)\right),
\end{align*} 
 which completes the proof. Remark that, for any fixed $\sigma > 1/2$ and any $n \in \mathbb{N},$ 
 $$\left| \frac{1-s}{s}\right|^n = \left( \frac{\sigma^2+t^2 - 2\sigma + 1}{\sigma^2+t^2}\right)^{\frac{n}{2}} =\left(\sqrt{1 - \frac{2\sigma - 1}{|s|^2}}\right)^{n} = 1 + O\left(\frac{1}{|s|^2}\right), $$
 as $|t| \to +\infty.$
 \end{proof}
 By using a similar reasoning one can prove, in general, that 
 \begin{theorem}
 If $\ell_n = \frac{\kappa}{n^{\alpha}} + o\left( \frac{1}{n^{\alpha}} \right),$ as $n \to +\infty,$ for an arbitrary real $\kappa$ and $\alpha > 1/2;$ then 
 $$\zeta(s) = \kappa \, \mathrm{Li}_{\alpha}\left(\frac{1-s}{s}\right) + o\left( \mathrm{Li}_{\alpha}\left(\sqrt{1 - \frac{2\sigma - 1}{|s|^2}}\right)\right), $$
 as $|t| \to + \infty .$ 
\end{theorem} 
\begin{proof}
In fact, if $\ell_n = \kappa / n^{\alpha} + o(n^{-\alpha})$ for $\alpha > 1/2$ then
\begin{align*}
\zeta(s) - \kappa \mathrm{Li}_{\alpha}\left( \frac{1-s}{s} \right) &= \frac{s}{s-1} + \ell_0 + \sum_{n\geq 1}\left( \ell_n - \frac{\kappa}{n^{\alpha}} \right)\left( \frac{1-s}{s}\right)^{n} \\
&= O(1) + o\left( \mathrm{Li}_{\alpha}\left(\sqrt{1 - \frac{2\sigma - 1}{|s|^2}}\right)\right).
\end{align*}
\end{proof} 

It should be noted that Theorem \ref{Ddiv} implies $\alpha \leq 1;$ because if $\alpha > 1$ then for a sufficiently large $n$ we have $|\ell_n| \sim 1/n^{\alpha};$ so that $\sum_{n >n_0} |\ell_n| $ (for some large $n_0$) is convergent which not the case thanks to Theorem \ref{Ddiv}.

Finally, we would like to provide a necessary condition for the truth of the Lindel\"of hypothesis. Namely, if the Laurent expansion of $(s-1)^k\zeta^k(s)$, for any positive integer $k$, is given near to $s=1$ by
\begin{equation}
(s-1)^k\zeta^k(s) = \sum_{m \geq 0}\frac{\lambda_{m,k}}{m!} (s-1)^m,
\label{klaurent}
\end{equation}
then we obtain the following

\begin{theorem}
If the Lindel\"{o}f hypothesis is true, then for any positive integer $k$, the series
 $$\sum_{n \geq -k}\ell_{n,k}^2$$
converges; where
 $$\ell_{n,k} := \begin{cases}\begin{array}{cl} \displaystyle (-1)^n \sum_{j=1}^n\binom{n-1}{j-1}\frac{\lambda_{j+k,k}}{(j+k)!} &\text{if} \quad n \geq 1, \\  \displaystyle (-1)^k \sum_{j=0}^{k+n}\binom{k-j}{-n} (-1)^j \frac{\lambda_{j,k}}{j!} &\text{if} \ -k \leq n \leq 0. \end{array}\end{cases} $$
\label{ThF}
\end{theorem}
\begin{proof}
 One can use a similar reasoning as in the proof of \cite[Th. 4.1]{Ste} to prove that the Lindel\"{o}f hypothesis is equivalent to the existence of the integrals
   
$$\int_{\mathbb{R}}\left|\zeta\left( \frac{1}{2} + it \right) \right|^{2k}\mathrm{d}\mu(t),$$
for any integer $k \geq 1$; that means that, if and only if the function $t \mapsto \zeta^k(1/2+it)$ belongs to the Hilbert space $\mathrm{L}^2(\mu)$, for any positive integer $k$. Thus, if we assume the truth of the Lindel\"of hypothesis then, for any positive integer $k,$ the following expansion holds in $\mathrm{L}^2(\mu)$

$$\zeta^k\left( \frac{1}{2}+it \right) = \sum_{n \in \mathbb{Z}}\ell_{n,k}e_n(t). $$ 
Consequently, Parseval's equality completes the proof. 

Concerning the explicit values of $(\ell_{n,k}),$ a similar reasoning as in the proof of Theorem \ref{Th1} yields
$$\ell_{n,k} = \frac{1}{2\pi i} \int_{\Re s = \frac{1}{2}}\frac{\zeta^k(s)}{s(1-s)}\left( \frac{s}{1-s} \right)^{n}\mathrm{d}s = -\mathrm{Res}_{s=1};$$
where $\mathrm{Res}_{s=1}$ is the residue of the integrand at the pole $s=1$ of order $k+n+1$ ($n \geq -k$). Hence, it is clear that $\ell_{n,k}=0$ for $n < -k$. Moreover, by using \eqref{klaurent}, we obtain
$$\mathrm{Res}_{s=1} = \lim_{s \to 1}\frac{(-1)^{n+1}}{(n+k)!}\sum_{m \geq 0} \frac{\lambda_{m,k}}{m!}\frac{\mathrm{d}^{n+k}}{\mathrm{d}s^{n+k}}\left[ s^{n-1}(s-1)^m\right].$$
Since, for $n \geq 1,$ 
$$\lim_{s\to 1}\frac{\mathrm{d}^{n+k}}{\mathrm{d}s^{n+k}}\left[ s^{n-1}(s-1)^m\right] = \binom{n-1}{m-k-1}(n+k)!, $$
where $ m = (k+1), \cdots ,(k+n);$ then
\begin{align*}
\mathrm{Res}_{s=1} &= (-1)^{n+1} \sum_{m=k+1}^{n+k} \binom{n-1}{m-k-1}\frac{\lambda_{m,k}}{m!}\\ 
&= (-1)^{n+1} \sum_{j=1}^{n} \binom{n-1}{j-1}\frac{\lambda_{j+k,k}}{(j+k)!}.
\end{align*}
Also, for $-k \leq n \leq 0;$ since

$$\lim_{s\to 1}\frac{\mathrm{d}^{n+k}}{\mathrm{d}s^{n+k}}\left[ s^{n-1}(s-1)^m\right] = (-1)^{k+n-m}\binom{k-m}{-n}(k+n)!, $$
where $m = 0, \cdots, (k+n);$ then
$$ \mathrm{Res}_{s=1} = (-1)^{k+1} \sum_{j=0}^{k+n}\binom{k-j}{-n} (-1)^j \frac{\lambda_{j,k}}{j!}.$$
\end{proof}
 
Last, but not least, the first author would like to express his belief on the truth of Lindel\"{o}f hypothesis thanks to last theorems. In the other hand, notice that the results obtained throughout this paper could be extended to a large class of functions, such as Dirichlet $L$ functions. 

\section*{Acknowledgments}
The authors are very grateful to the anonymous reviewr for her/his instructions and valuable comments and suggestions.


\begin{thebibliography}{00}

\bibitem{Alk}
E. Alkan, Resolution of a conjecture on the convexity zeta functions, J. Math. Anal. Appl. 472 (2019), 1987--2016.

\bibitem{Bal}
M. Balazard, E. Saias, M. Yor, Notes sur la fonction $\zeta$ de Riemann, 2. Adv. Math. 143 (1999) 284--287.
 
  
\bibitem{Ber}
B.C. Berndt, On the Hurwitz zeta-function. Rocky Mountain J. Math. 2 (1972) 151--157.

\bibitem{Bir}
G.D. Birkhoff, Proof of the ergodic theorem. Proc. Natl. Acad. Sci. USA 17(12) (1931) 656--660.


\bibitem{Bour}
J. Bourgain, Decoupling, exponential sums and the Riemann zeta function, J. Amer. Math. Soc. 30 (2017) 205--224.

\bibitem{Brig}
W. E. Briggs, Some constants associated with the Riemann zeta function, Michigan
Math. J. 3 (1955-1956), 117--121.

\bibitem{Bui}
H.M. Bui, J.B. Conrey, M.P. Young, More that $41\%$ of the zeros of the zeta function are on the critical line, {http://arxiv.org/abs/1002.4127v2}.

\bibitem{Cof}
M.W. Coffey, Evaluation of some second moment and other integrals for the Riemann, Hurwitz, and Lerch zeta functions, arXiv:1101.5722v1 [math-ph] (2011).

\bibitem{Con}
J.B. Conrey, More than two fifths of the zeros of the Riemann zeta function are on the critical line, J. Reine Angew. Math. 399 (1989) 1--26.

\bibitem{Edw}
H.M. Edwards, Riemann's zeta function, Academic Press, New York (1974)

\bibitem{Ela}
L. Elaissaoui, Z. E. Guennoun, On logarithmic integrals of the Riemann zeta-function and an approach to the Riemann Hypothesis by a geometric mean with respect to an ergodic transformation, European Journal of Math. 1 (2015) 829--847.

\bibitem{Fen}
S. Feng, Zeros of the Riemann zeta function on the critical line, J. Number Theory, 132(4) (2012) 511--542.

\bibitem{Gar}
J.B. Garnett, Bounded analytic functions, Academic Press, New York (1981).  

\bibitem{Jen}
R. Jentzsch, Untersuchungen zur Theorie der Folgen analytischer Funktionen. Acta Math. 41 (1916), 219--251.

\bibitem{Khin}
A.Yu. Khintchine, Zu Birkhoffs L\"osung des Ergodenproblems. Math. Ann. 107(1) 485-488 (1933)

\bibitem{Lind}
E. Lindel\"{o}f, Quelques remarques sur la croissance de la fonction $\zeta(s)$, Bull. sci. math. 32 (1908) 341--356. 

\bibitem{Lev}
N. Levinson, More than one third of the zeros of Riemann's zeta function are on $\sigma = 1/2$, Adv. Math. 13 (1974) 383--436.
 

\bibitem{Rud}
W. Rudin, Real and Complex Analysis, third ed., McGraw-Hill Book Company, New York 1987.


\bibitem{Ste}
J. Steuding, Sampling the Lindel\"{o}f hypothesis with an ergodic transformation, RIMS Kokyuroku Bessatsu. B34 (2012) 361-381.

\bibitem{Titch}
E.C. Titchmarsh, The Theory of the Riemann Zeta-Function, 2nd ed., Oxford University Press, New York, 1986.
 
\bibitem{Titch2}
E.C. Titchmarsh, The Theory of Functions,  2nd ed., Oxford University Press, Oxford, 1939.

\end{thebibliography}
\end{document}